\documentclass[12pt]{article}
\usepackage{amssymb,amsfonts,amsmath,latexsym,epsf,url}
\usepackage[usenames,dvipsnames]{pstricks}
\usepackage{pstricks-add}
\usepackage{subfigure}
\usepackage{graphicx}
\usepackage{color,soul}
\usepackage{epsfig}
\usepackage{tikz}
\usepackage[colorlinks]{hyperref}
\usepackage{cite}
\usepackage{algorithm}
\usepackage{algpseudocode}
\usepackage{graphicx}
\usepackage{booktabs}
\usepackage{multirow}

\newtheorem{theorem}{Theorem}[section]

\newtheorem{corollary}[theorem]{Corollary}

\newtheorem{remark}[theorem]{Remark}
\newtheorem{example}[theorem]{Example}

\newcommand{\qed}{\hfill $\square$\medskip}

\begin{document}

\def\nt{\noindent}

\title{Totient-based Identities for Integer Pairs with Equal Divisor Sums}

\author{
	Ali Reza Mavaddat$^1$ \and Saeid Alikhani$^2$\footnote{Corresponding author}
}


\maketitle

\begin{center}
	$^1$Independent researcher, Abdol Motalleb Blvd, Mashhad, Iran\\
	 {\tt alireza.mavaddat@gmail.com}

\medskip
$^2$Department of Mathematical Sciences, Yazd University, Yazd, Iran\\
{\tt alikhani206@gmail.com}

\end{center}
	
\begin{abstract}
	A pair of integers is called amicable if each number equals the sum of the proper divisors of the other. In this paper, we investigate a broader class of integer pairs $(A, B)$ that satisfy the weaker condition $\sigma(A) = \sigma(B)$, where $\sigma$ is the sum-of-divisors function. We focus on pairs whose greatest common divisor is a power of two and analyze their distinct prime factorizations. Specifically, we consider configurations of the forms $(A=2^n ab,\ B=2^n cd)$, $(A=2^n abc,\ B=2^n de)$, and $(A=2^n abc,\ B=2^n def)$, where the factors are odd primes. For each configuration, we derive explicit symmetric identities that must be satisfied. These identities relate the odd prime factors of $A$ and $B$ to the quantity $\varphi(A)+\varphi(B)$, where $\varphi$ denotes Euler’s totient function. As a consequence, these relations provide necessary structural constraints for amicable pairs of these specific forms, while also revealing an underlying algebraic symmetry for any pair with equal divisor sums.
\end{abstract}


\medskip
\noindent{\bf AMS Subj.\ Class.:}  11A25, 11N25.


	\section{Introduction}
	
		One of the most remarkable sums in number theory is the aliquot sum $s(n)$ of a positive integer $n$. It is the sum of all proper divisors of $n$, that is, all divisors of $n$ other than $n$ itself. So
	\[
	s(n)=\sum_{d|n, d<n} d.
	\]
	
	As usual $\sigma(n)$ is the sum of all the divisors of $n$. Thus $\sigma(n)=s(n)+n$. 
	 As usual the abbreviations gcd and lcm stand for the greatest common divisor and the least common multiple. 	
	Two natural numbers $m$ and $n$ are said to be amicable, if
	$\sigma(m)=\sigma(n)=m+n$. 	
	If $m=n$, then $m$ is perfect. It is well known that the smallest perfect number is $6$. The smallest pair of amicable numbers, $(220,284)$, was known to the Pythagoreans. Although more than 2000 years have passed, it is still unknown whether infinitely many amicable pairs exist.
		
	The first significant rule established for finding amicable pairs is attributed to Thabit ibn Qurra, a 9th-century mathematician (826--901 AD), which is as follows:
	
	\[ p = 3 \cdot 2^{n-1} - 1, \]
	\[ q = 3 \cdot 2^n - 1, \]
	\[ r = 9 \cdot 2^{2n-1} - 1, \]
	\[ (M = 2^n \cdot p \cdot q \quad , \quad N = 2^n \cdot r).
	 \]
		Using Thabit ibn Qurra's rule, only three amicable pairs have been discovered so far. The first pair $(220, 284)$ is attributed to the Pythagoreans (300 BC).
		The pair $(17296, 18416)$ was first discovered by Ibn al-Banna al-Marrakushi in the 13th century AD. Later, in 1636 AD, Fermat announced in a letter to Mersenne that the numbers $(18416, 17296)$ are an amicable pair \cite{Costello}.
		The pair $(9363584, 9437056)$ was discovered in the 17th century by the Iranian mathematician Muhammad Baqir Yazdi, though it is attributed to Descartes in European texts \cite{Borho}.
	
	Thabit ibn Qurra's formula only yields results for:
	\[
	 n = 2, 4, 7,
	  \]
	and a fourth pair has not yet been discovered using this formula.
		In the decades between 1700 and 1710, the great Euler presented a list of 64 amicable pairs. It was later realized in 1909 and 1914 that two of these pairs were not actually amicable.
		Euler provided a more general formula than Thabit ibn Qurra's rule for finding amicable pairs; in fact, Thabit ibn Qurra's formula is a special case of Euler's formula. 	
	The formula is as follows:
	\[ A = 2^n \cdot p \cdot q, \]
	\[ B = 2^n \cdot r, \]
	\[ p = 2^m \cdot (2^{n-m} + 1) - 1, \]
	\[ q = 2^n \cdot (2^{n-m} + 1) - 1, \]
	\[ r = 2^{n+m} \cdot (2^{n-m} + 1)^2 - 1.
	\]
		All three numbers $p, q, r$ must be prime numbers.

	The concept of amicable numbers has undergone various generalizations since the early 20th century. In 1913, Dickson \cite{5} defined a set of natural numbers $\{n_1, n_2, n_3\}$ as an {amicable triple} if:
	\begin{equation*}
	\sigma(n_1) = \sigma(n_2) = \sigma(n_3) = n_1 + n_2 + n_3.
	\end{equation*}
	Equivalently, using the aliquot sum function $s(n) = \sigma(n) - n$, this relationship can be expressed as:
	\begin{equation*}
	\begin{cases} 
	s(n_1) = n_2 + n_3, \\ 
	s(n_2) = n_1 + n_3, \\ 
	s(n_3) = n_1 + n_2 .
	\end{cases}
	\end{equation*}
	The smallest triple satisfying Dickson's definition is $(1980, 2016, 2556)$. Dickson also extended this definition to $k$-tuples.
	
	Later, Yanney \cite{19} proposed an alternative definition for amicable triples, requiring that:
	\begin{equation*}
	2\sigma(n_1) = 2\sigma(n_2) = 2\sigma(n_3) = n_1 + n_2 + n_3.
	\end{equation*}
	This is equivalent to the system:
	\begin{equation*}
	\begin{cases} 
	n_1 = s(n_2) + s(n_3), \\ 
	n_2 = s(n_1) + s(n_3), \\ 
	n_3 = s(n_1) + s(n_2). 
	\end{cases}
	\end{equation*}
	Under Yanney's criteria, the smallest amicable triple is $(238, 255, 371)$. Similar to Dickson, Yanney generalized this to $k$-tuples.

	Other researchers explored multiply or feeble variations. Carmichael \cite{3} defined a {multiply amicable pair} $(m, n)$ such that:
	\begin{equation*}
	\sigma(m) = \sigma(n) = t(m + n),
	\end{equation*}
	where $t$ is a positive integer. Mason \cite{12} further extended this logic to $k$-tuples. In 1995, Cohen, Gretton, and Hagis \cite{4} introduced the concept of {multiamicable numbers}, where the aliquot sum of each number is a multiple of the other. Formally, a pair $(m, n)$ is $(\alpha, \beta)$-amicable if:
	\begin{equation*}
	s(m) = \alpha n \quad \text{and} \quad s(n) = \beta m.
	\end{equation*}
	
	Bishop, Bozarth, Kuss, and Peet \cite{2} introduced {feebly amicable $k$-tuples}, defined as a set of numbers $\{n_1, \dots, n_k\}$ satisfying:
	\begin{equation*}
	\sum_{i=1}^{k} \frac{n_i}{\sigma(n_i)} = 1.
	\end{equation*}
	Every standard amicable $k$-tuple satisfies this identity. For $k=2$, the smallest feebly amicable pair is $(4, 12)$.

	An integer $n$ is considered an amicable number if it is a member of an amicable pair, which is equivalent to the condition:
	\begin{equation*}
	\sigma(\sigma(n) - n) = \sigma(n).
	\end{equation*}
	Let $A(x)$ denote the counting function for amicable numbers up to $x$. In 1955, Erd\" os \cite{6} proved that amicable numbers have an asymptotic density of zero:
	\begin{equation*}
	\lim_{x \to \infty} \frac{A(x)}{x} = 0.
	\end{equation*}
	Subsequent research \cite{7, 14, 15, 17} has established increasingly tight upper bounds for $A(x)$. Currently, the best known bound is provided by Pomerance \cite{16}:
	\begin{equation*}
	A(x) \le x \cdot \exp\left( - \sqrt{\ln x \ln \ln x} \right).
	\end{equation*}

	Numerous studies have been dedicated to extending these concepts (see \cite{1, 8, 9, 10, 11, 13}). Specifically, recently, Dimitrov \cite{Dimit} proposed and studied  $(\alpha_1,\dots,\alpha_k)$-multiamicable $k$-tuples which is a new generalizations of amicable numbers.  Motivated by these historical developments, this paper proposes new structural results for pairs of integers with equal divisor sums. If we examine the prime factorization of amicable number pairs, we find that there is no consistent pattern defined for all pairs; some pairs have two prime factors, some have three, and others have power-based factors, etc. In this paper, we consider pairs that have a greatest common divisor (gcd) of the form $2^n$. To obtain stronger and more general identities, we initially relax the condition of being an amicable pair and only require that $\sigma(A) = \sigma(B)$. We then derive necessary conditions for such pairs. As a direct consequence, these identities apply to amicable pairs of the specified forms. Furthermore, we demonstrate that these conditions are necessary but not sufficient for amicability.

	\section{Main results}
	While millions of pairs are known today through computer algorithms, no single formula generates all of them. This paper proposes a property based on the prime factorization of pairs where the gcd is a power of 2. 
	 
	\subsection{Relationship for two-term factors}
	\begin{theorem}\label{th1}
		Let $n\in \mathbb{N}$, let $a,b,c,d$ be odd prime numbers, and let
		\[
		A=2^n ab \quad \text{and} \quad B=2^n cd.
		\]
		If $\sigma(A)=\sigma(B)$, then
		\[
		a+b = 2^{-n}\bigl(B-(\varphi(A)+\varphi(B))\bigr)+1
		\]
		and
		\[
		c+d = 2^{-n}\bigl(A-(\varphi(A)+\varphi(B))\bigr)+1.
		\]
	\end{theorem}
	
	\begin{proof}
		We establish the formula for $a+b$; the derivation of the formula for $c+d$ is symmetrical.
		
		The hypotheses imply that
		\[
		\varphi(A)=2^{n-1}(a-1)(b-1)
		\]
		and
		\[
		\varphi(B)=2^{n-1}(c-1)(d-1),
		\]
		and that
		\[
		\sigma(A)=(2^{n+1}-1)(a+1)(b+1)
		\]
		and
		\[
		\sigma(B)=(2^{n+1}-1)(c+1)(d+1).
		\]
		
		Furthermore, since $A,B$ form an amicable pair we have
		\[
		\sigma(A)=\sigma(B),
		\]
		so
		\[
		(a+1)(b+1)=(c+1)(d+1).
		\]
		
		Now we have
		\begin{align*}
			B-\varphi(B)
			&=2^ncd-2^{n-1}(c-1)(d-1) \\
			&=2^{n-1}(2cd-cd+c+d-1) \\
			&=2^{n-1}\bigl((c+1)(d+1)-2\bigr) \\
			&=2^{n-1}\bigl((a+1)(b+1)-2\bigr)
			\qquad \text{since } \sigma(A)=\sigma(B) \\
			&=2^{n-1}\bigl((a-1)(b-1)+2(a+b)-2\bigr) \\
			&=\varphi(A)+2^n(a+b)-2^n.
		\end{align*}
		
		Hence
		\[
		2^n(a+b)=B-\varphi(B)-\varphi(A)+2^n,
		\]
		which immediately implies the given formula.
		\qed
	\end{proof}

By Theorem \label{th1}, we have the following corollary for an amicable pair:

\begin{corollary}\label{cor1}
	Let $n\in \mathbb{N}$, let $a,b,c,d$ be odd prime numbers, and let
	\[
	A=2^n ab \quad \text{and} \quad B=2^n cd.
	\]
	If $A,B$ form an amicable pair, then
	\[
	a+b = 2^{-n}\bigl(B-(\varphi(A)+\varphi(B))\bigr)+1
	\]
	and
	\[
	c+d = 2^{-n}\bigl(A-(\varphi(A)+\varphi(B))\bigr)+1.
	\]
\end{corollary}
\begin{example}
Let us apply Theorem \ref{th1} to  pair $(A=2620,B=2924)$. We have
\[
A=2620=2^2\times 5 \times 131, ~~~~~B=2924=2^2\times 17\times 43.
\]
We also have $ \varphi(2620) +  \varphi(2924) = 2384$. Since $2620-2384 = 236 = 2^2\times 59$ and $2924-2384 = 540 = 2^2\times 3^3\times5$, so  
\[
2^{-2}\big(B-(\varphi(2620) +  \varphi(2924))\big)+1=3^3\times 5+1,
\]
which is equal to $a+b=5+131=136$. Similarly 
\[
2^{-2}\big(A-(\varphi(2620) +  \varphi(2924))\big)+1=60,
\]
which is equal to $c+d=17+43=60.$
\end{example}

\begin{remark}
The converse of Corollary \ref{cor1} is not true. For example, take
	\[
	A = 230 = 2 \cdot 5 \cdot 23, \qquad B = 238 = 2 \cdot 7 \cdot 17.
	\]
	These numbers satisfy \(\gcd(A, B) = 2\) and
	\[
	\sigma(A) = \sigma(B) = 432, \qquad \varphi(A) + \varphi(B) = 88 + 96 = 184.
	\]
	Thus both identities of Theorem \ref{th1} hold:
	\[
	\frac{238 - 184}{2} + 1 = 28 = 5 + 23, \qquad \frac{230 - 184}{2} + 1 = 24 = 7 + 17.
	\]
	Nevertheless, \(A + B = 468 \neq 432\), so $A$ and $B$ do not form an amicable pair.
	
\end{remark}

\subsection*{Relationship for two-term and three term factors}

\begin{theorem} \label{thm23}
Suppose that 
\begin{align*}
A &= 2^n \cdot a \cdot b \cdot c, \\
B &= 2^n \cdot d \cdot e,
\end{align*}
where $a, b, c, d, e$ are odd primes and $\sigma(A)=\sigma(B)$. 
Then we have 
\begin{align*}
c(a+b)+ab&=2^{-n}\big(B-(\varphi(A) + \varphi(B))\big), 
\end{align*}
\begin{align*} \label{}
(d+e)-(a+b+c)&=2^{-n}\big(A-(\varphi(A) + \varphi(B))\big).
\end{align*}
\end{theorem} 
\begin{proof}
We have:
\[ 
\sigma(A) = (2^{n+1}-1)(a+1)(b+1)(c+1), ~~  \sigma(B)=(2^{n+1}-1)(d+1)(e+1).
\]
	Set 	
		\[
		P = abc, \quad Q = ab + ac + bc, \quad S = a + b + c, \quad R = de, \quad T = d + e,
		\]
		equality of divisor sums gives \(R - P = Q + S - T\). If \(F = \varphi(A) + \varphi(B)\), then
		\[
		2^{-n}F = \tfrac{1}{2}(P - Q + S + R - T),
		\]
		and hence
		\[
		2^{-n}(B - F) = Q, \qquad 2^{-n}(A - F) = T - S.
		\]
		This calculation proves theorem. \qed
	\end{proof}
	
	\begin{corollary} \label{cor23}
		Let $(A, B)$ be an amicable pair of the form
		\begin{align*}
			A &= 2^n \cdot a \cdot b \cdot c, \\
			B &= 2^n \cdot d \cdot e,
		\end{align*}
		where $a, b, c, d, e$ are odd primes. 
		Then we have 
		\begin{align*}
			c(a+b)+ab&=2^{-n}\big(B-(\varphi(A) + \varphi(B))\big), 
		\end{align*}
		\begin{align*} \label{}
			(d+e)-(a+b+c)&=2^{-n}\big(A-(\varphi(A) + \varphi(B))\big).
		\end{align*}
	\end{corollary}

	\begin{example}
		Let us apply Theorem \ref{thm23} to amicable pair $(A=185368,B=203432)$. We have
		\[
		A=185368=2^3\times 17\times 29\times 47, ~~~~~B=203432=2^3\times 59\times 431.
		\]
		We also have $ \varphi(A) +  \varphi(B) =  182192 $. Since $185368-182192 = 3176 = 2^3\times 397$ and $203432-182192 = 21240 = 2^3\times 3^2 \times 5\times 59$, so  
		\[
		2^{-3}\big(A-(\varphi(A) +  \varphi(B))\big)=397,
		\]
		which is equal to $(d +e)-(a+b+c)=(59+431)-(17+29+47)=397$. Similarly 
		\[
		2^{-3}\big(B-(\varphi(A) +  \varphi(B))\big)=2655,
		\]
		which is equal to $c(a +b) +ab=47\times 46+17\times 29.$
	\end{example}
	
	\subsection{Relationship for three-term factors}

								\begin{theorem}\label{thm33}
									Let $(A, B)$ be an amicable pair of the form
									\[
									A = 2^n a b c, \qquad B = 2^n d e f,
									\]
									where $a,b,c,d,e,f$ are odd primes.
									Then
									\begin{align}
									2^{-n}\bigl(A - (\varphi(A)+\varphi(B))\bigr) - 1 &= f(d+e) + de - (a+b+c), \tag{1} \\
									2^{-n}\bigl(B - (\varphi(A)+\varphi(B))\bigr) - 1 &= c(a+b) + ab - (d+e+f). \tag{2}
									\end{align}
								\end{theorem}
															
								\begin{proof}		
														Define:
								\[
								S_A=a+b+c,~ S_B=d+e+f,~ P_A=abc,~ P_B=def,~ Q_A=ab+ac+bc,~ Q_B=de+df+ef.
								\]
															We have the following Euler's totient:
								\[
								\varphi(A)=2^{n-1}(a-1)(b-1)(c-1)=2^{n-1}(P_A-Q_A+S_A-1),
								\]
								\[
								\varphi(B)=2^{n-1}(P_B-Q_B+S_B-1).
								\]
								
								Then:
								\begin{align*}
									A-(\varphi(A)+\varphi(B)) &= 2^nP_A - 2^{n-1}\big[P_A+P_B-(Q_A+Q_B)+(S_A+S_B)-2\big] \\
									&= 2^{n-1}\big[P_A-P_B+Q_A+Q_B-S_A-S_B+2\big].
								\end{align*}
								
								Multiply by $2^{-n}$ and subtract $1$:
								\[
								2^{-n}\big[A-(\varphi(A)+\varphi(B))\big]-1 = \frac{1}{2}\big[P_A-P_B+Q_A+Q_B-S_A-S_B\big]. \tag{L1}
								\]
								
								Similarly,
								\[
								2^{-n}\big[B-(\varphi(A)+\varphi(B))\big]-1 = \frac{1}{2}\big[P_B-P_A+Q_A+Q_B-S_A-S_B\big]. \tag{L2}
								\]
								
								Now for the right hand side of (1) and (2) we have,
								\[
								\text{RHS}(1) = f(d+e)+de - S_A = Q_B - S_A,
								\]
								\[
								\text{RHS}(2) = c(a+b)+ab - S_B = Q_A - S_B.
								\]
								
								We show (1): $L1 \stackrel{?}{=} Q_B - S_A$. Multiply both by $2$:
								\[
								P_A-P_B+Q_A+Q_B-S_A-S_B = 2Q_B-2S_A
								\iff P_A-P_B = Q_B-Q_A + S_B-S_A. \tag{*}
								\]
								
								For (2): $L2 \stackrel{?}{=} Q_A - S_B$, similarly:
								\[
								P_B-P_A+Q_A+Q_B-S_A-S_B = 2Q_A-2S_B
								\iff P_A-P_B = Q_B-Q_A + S_B-S_A.
								\]
								
								Thus both (1) and (2) hold if (*) holds.
								
							Let have verification of (*) from amicability. 								
								From $\sigma(A)=A+B$:
								\[
								(2^{n+1}-1)(P_A+Q_A+S_A+1) = 2^n(P_A+P_B).
								\]
								Similarly from $\sigma(B)=A+B$:
								\[
								(2^{n+1}-1)(P_B+Q_B+S_B+1) = 2^n(P_A+P_B).
								\]
								
								Subtract the second from the first and let $D_P=P_A-P_B$, $D_Q=Q_A-Q_B$, $D_S=S_A-S_B$:
								\[
								(2^{n+1}-1)(D_P+D_Q+D_S)=0.
								\]
								Since $2^{n+1}-1 \neq 0$, we have $D_P+D_Q+D_S=0$, i.e.,
								\[
								P_A-P_B = Q_B-Q_A + S_B-S_A,
								\]
								which is (*). Hence (1) and (2) hold.\qed
								
							\end{proof}
													
			\begin{example}
			Let us apply Theorem \ref{thm23} to amicable pair  
					$$A = 26\,989\,290\,624\,832, \; B = 26\,730\,182\,367\,808.$$
				We have the following factorization:
				\[
				A = 2^6 \times 487 \times 1021 \times 848119, \qquad
				B = 2^6 \times 139 \times 181 \times 16\,600\,783.
				\]
				
				Assign:
				\[
				a = 487, \quad b = 1021, \quad c = 848119, \qquad
				d = 139, \quad e = 181, \quad f = 16\,600\,783, \quad n = 6.
				\]
			By notation in the proof of Theorem \ref{thm33}, 	
							\[
				S_A = a+b+c = 487+1021+848119 = 849627,
				\]
				\[
				S_B = d+e+f = 139+181+16\,600\,783 = 16\,601\,103,
				\]
				\[
				Q_A = ab+ac+bc = 1\,279\,460\,679, \qquad 
				Q_B = de+df+ef = 5\,312\,275\,719.
				\]
			Now we compute \(\varphi(A)\) and \(\varphi(B)\):
				\[
				\varphi(A) = 2^{5}(a-1)(b-1)(c-1) 
				= 32 \times 486 \times 1020 \times 848118 = 13\,453\,729\,758\,720,
				\]
				\[
				\varphi(B) = 2^{5}(d-1)(e-1)(f-1)
				= 32 \times 138 \times 180 \times 16\,600\,782 = 13\,195\,629\,596\,160.
				\]
				\[
				\varphi(A)+\varphi(B) = 26\,649\,359\,354\,880.
				\]
				Here let to check formula (1): 
				\[
				\text{LHS}(6) = 2^{-6}\bigl[A - (\varphi(A)+\varphi(B))\bigr] - 1.
				\]
				\[
				A - (\varphi(A)+\varphi(B)) = 339\,931\,269\,952, \quad
				2^{-6} = \frac{1}{64}.
				\]
				\[
				\frac{339\,931\,269\,952}{64} = 5\,311\,426\,093, \quad
				\text{LHS} = 5\,311\,426\,093 - 1 = 5\,311\,426\,092.
				\]
				
				\[
				\text{RHS} = f(d+e) + de - (a+b+c) = Q_B - S_A.
				\]
				\[
				Q_B - S_A = 5\,312\,275\,719 - 849\,627 = 5\,311\,426\,092.
				\]
				\[
			{\text{LHS} = \text{RHS} = 5\,311\,426\,092.}
				\]
				Now let check the formula (2):
				\[			
				\text{LHS}(7) = 2^{-6}\bigl[B - (\varphi(A)+\varphi(B))\bigr] - 1.
				\]
				\[
				B - (\varphi(A)+\varphi(B)) = 80\,823\,012\,928.
				\]
				\[
				\frac{80\,823\,012\,928}{64} = 1\,262\,859\,577, \quad
				\text{LHS} = 1\,262\,859\,577 - 1 = 1\,262\,859\,576.
				\]
				
				\[
				\text{RHS} = c(a+b) + ab - (d+e+f) = Q_A - S_B.
				\]
				\[
				Q_A - S_B = 1\,279\,460\,679 - 16\,601\,103 = 1\,262\,859\,576.
				\]
				\[
				{\text{LHS} = \text{RHS} = 1\,262\,859\,576.}
				\]

				Both equations (1) and (2) of Theorem \ref{thm33} hold exactly for the given amicable pair.

	\end{example}

		\section{The general calculation}
		
	Based on the patterns observed in 2-term and 3-term factors, we propose the
	following general rules.

	\begin{theorem}
		Suppose
		\[
		A = 2^n \prod_{i=1}^k p_i, \qquad B = 2^n \prod_{j=1}^m q_j,
		\]
		where \(n \ge 1\), the factors are odd primes distinct within each product, and \(\sigma(A) = \sigma(B)\). Let \(e_r(p)\) denote the elementary symmetric polynomial of degree \(r\) in the \(p_i\), with \(e_0(p) = 1\). Put
		\[
		H_p = \sum_{\substack{\ell \ge 0 \\ k-1-2\ell \ge 0}} e_{k-1-2\ell}(p), \qquad J_p = \sum_{\substack{\ell \ge 0 \\ k-2-2\ell \ge 0}} e_{k-2-2\ell}(p),
		\]
		and define \(H_q, J_q\) in the same way. Set
		\[
		x = 2^{-n} \big( B - \varphi(A) - \varphi(B) \big), \qquad y = 2^{-n} \big( A - \varphi(A) - \varphi(B) \big).
		\]
		Then
		\begin{equation}
			x = H_p - J_q, \qquad y = H_q - J_p. \tag{3}
		\end{equation}
		\end{theorem} 
		\begin{proof}		
		Let
		\[
		U_p = \prod_i (p_i + 1), \qquad T_p = \prod_i (p_i - 1).
		\]
		Expansion gives
		\[
		U_p - T_p = 2H_p, \qquad U_p = \prod_i p_i + H_p + J_p.
		\]
		Equality of divisor sums gives \(U_p = U_q = U\), while
		\[
		2^{-n} \big( \varphi(A) + \varphi(B) \big) = \frac{1}{2}(T_p + T_q) = U - H_p - H_q.
		\]
		Substitution proves (3).
		\end{proof}

	\section{Conclusion}
	
	In this paper, we considered a pair $(A,B)$ satisfying the condition $\sigma(A)=\sigma(B)$ based on their prime factorization when the greatest common divisor is a power of two. By focusing on pairs of the forms \(2^{n}ab\) and \(2^{n}cd\), \(2^{n}abc\) and \(2^{n}de\), and \(2^{n}abc\) and \(2^{n}def\), we derived explicit relations that connect the sum of Euler’s totient functions \(\varphi(A)+\varphi(B)\) with the prime factors of \(A\) and \(B\).
	
	While these results were originally motivated by amicable numbers (where $\sigma(A) = A+B$), our approach reveals that the algebraic symmetries hold under the strictly weaker condition $\sigma(A)=\sigma(B)$. This generalization not only provides necessary conditions for amicability in these specific forms but also opens a pathway to study the interplay between Euler's totient function and the sum-of-divisors function for a broader class of integer pairs. Future work could focus on characterizing the exact conditions under which these necessary identities become sufficient for amicability.

	\section*{Declaration of Competing Interest} 
	The authors declare that they have no known competing financial 
	interests or personal relationships that could have appeared to influence 
	the work reported in this paper.
	
	\section*{Conflict of interest}
	There are none.
	\section*{Data Availability}
	There are none used.

\end{document}